\documentclass[11pt]{amsart}
\usepackage{amssymb,latexsym}
\usepackage{pstricks,pst-grad,pst-char,epsfig, graphicx}
\newtheorem{theorem}{Theorem}[section]

\newtheorem{lemma}[theorem]{Lemma}
\newtheorem{remark}[theorem]{Remark}

\newcommand{\Z}{{\mathbb{Z}}}

\begin{document}

\title[Donaldson-Thomas invariants]
{Local Donaldson-Thomas invariants of Blowups of surfaces }

\author[Jianxun Hu ]{Jianxun Hu$^*$   }
\address{Department of Mathematics\\ Sun Yat-sen
University\\Guangzhou 510275\\ P. R. China}
\email{stsjxhu@mail.sysu.edu.cn}
\thanks{${}^*$Partially supported by the NSFC Grant 10825105 }


\begin{abstract}
Using the degeneration formula for Doanldson-Thomas invariants, we
proved a formula for the change of Donaldson-Thomas invariants of
local surfaces under blowing up along points.

\end{abstract}

\maketitle
\date{}

\section{\bf Introduction}
Given a smooth projective Calabi-Yau $3$-fold $X$, the moduli
space of stable sheaves on $X$ has virtual dimension zero.
Donaldson and Thomas \cite{D-T} defined the holomorphic Casson
invariant of $X$ which essentially counts the number of stable
bundles on $X$. However, the moduli space has positive dimension
and is singular in general. Making use of virtual cycle technique
(see \cite{B-F} and \cite{L-T}), Thomas in \cite{Thomas} showed
that one can define a  virtual moduli cycle for some $X$ including
Calabi-Yau and Fano $3$-folds. As a consequence, one can define
Donaldson-type invariants of $X$ which are deformation invariant.
Donaldson-Thomas invariants provide  a new  vehicle to study the
geometry and other aspects of higher-dimensional varieties. It is
important to understand these invariants.

It is well-known \cite{MNOP1, MNOP2}that there is a correspondence
between Donaldson-Thomas invariants and Gromov-Witten invariants.
Both invariants are deformation independent. On the side of
Gromov-Witten invariants, Li and Ruan in \cite{L-R} first
established the degeneration formula of Gromov-Witten invariants
in symplectic geometry. J. Li proved an algebraic geometry version
of this degeneration formula. In \cite{Hu1, Hu2}, the author
studied the change of Gromov-Witten invariants under the blowup.
The author \cite{Hu3} also studies the change of local
Gromov-Witten invariants of Fano surfaces under the blowup.  In
the birational geometry of $3$-folds, we have blowups and flops
which are semistable degenerations. In \cite{HL} the authors
studied how Donaldson-Thomas invariants change under the blowup at
a point, some flops and extremal transitions.

Local del Pezzo surface used to play an important role in physics.
Local de Pezzo surfaces are usually associated to phase
transitions in the K\"ahler moduli space of various string,
M-theory, and F-theory compactifications.Non-toric del Pezzo
surfaces seem to be related to exotic physics in four, five and
six dimensions such as nontrivial fixed points of the
renormalization group \cite{GMS} without lagrangian description
and strongly interacting noncritical strings. There is also a
relation between non-toric del Pezzo surfaces and string junctions
in F-theory \cite{KMV,LMW}. Certain problems of physical interest
such as counting of BPS states reduce to questions related to
topological strings on local del Pezzo surfaces. In this paper, we
will use the degeneration formula for Donaldson-Thomas invariants
to study the change of the Donaldson-Thomas invariants of the
local surface under the blowup.

Let $S$ be a smooth surface and $K_S$ its canonical bundle. Denote
by $Y_S= {\mathbb P}(K_S\oplus {\mathcal O})$ the projective
bundle completion of the total space of the canonical bundle
$K_S$. The Donaldson-Thomas theory of $Y_S$ is well defined in
every rank. Let $\gamma_i\in H^*(Y_S)$, $i=1,\cdots, r$. Denote by
$\tilde{\tau}_{k_i}(\gamma_i)$ the associated descendent fields in
Donaldson-Thomas theory, which is defined in \cite{MNOP2}. For
$\beta\in H_2(Y_S,{\mathbb Z})$ and an integer $n\in {\mathbb Z}$,
denote by $\langle \tilde{\tau}_{k_1}(\gamma_1),\cdots,
\tilde{\tau}_{k_r}(\gamma_{r})\rangle^{Y_S}_{n,\beta}$ the
descendent Donaldson-Thomas invariant of $Y_S$. Denote by
$Z'_{DT}(S;q)_\beta$ the reduced partition function for the
Donaldson-Thomas theory of the local Calabi-Yau geometry of $S$.

Denote by $p:\tilde{S}\longrightarrow S$ the natural projection of
the blow-up of $S$ at a smooth point $p_0\in S$. Let $\beta\in
H_2(S,\mathbb{Z})$ and $p!(\beta) = PDp^*PD(\beta)\in
H_2(\tilde{S}, \mathbb{Z})$. In \cite{Hu3}, we use the
degeneration formula to study the change of local Gromov-Witten
invariants under the blowup of the Fano surfaces. Similarly, we
observed that the Donaldson-Thomas invariants of $Y_S$ of degree
$\beta$ is equal to the Donaldson-Thomas invariants of
$Y_{\tilde{S}}$ of degree $p!(\beta)$.

 We can find a sequence of birational threefolds all of
whose invariants are equal. In fact, the birational threefolds are
the projective completion $Y_S$ of $K_S$, the blow-up
$\tilde{Y}_S$ of $Y_S$ along the fiber over $p_0$, the projective
completion $Y_{\tilde{S}}$ of $K_{\tilde{S}}$ and $Z$, a threefold
dominating the last two, obtained by blowing them up along a
specific section of the exceptional divisor in $\tilde{S}$. For
each pair of spaces, a degeneration is constructed with the goal
of comparing absolute invariants of one with relative invariants
of the other. Then we prove that the virtual dimension of one of
the moduli spaces of relative stable maps appearing in the
degeneration formula is negative as soon as there are nontrivial
contacts with the relative divisors. Next a second application of
the degeneration formula compares such relative invariants with
the absolute invariants of the same space. This sequence of
comparing results implies the following theroem:

\begin{theorem}\label{thm-1-1}
Suppose that $S$ is a smooth surface and $\tilde{S}$ is the
blown-up surface of $S$ at a smooth point $p$. Let $\beta\in
H_2(S,\mathbb{Z})$. Then we have
\begin{equation}
    Z'_{DT}(S;q)_\beta = Z'_{DT}(\tilde{S};q)_{p!(\beta)},
\end{equation}
where $p:\tilde{S}\longrightarrow S$ is the natural projection of
the blowup.
\end{theorem}

\begin{remark}
Theorem \ref{thm-1-1} make it possible to compute the
Donaldson-Thomas invariants of local nontoric del Pezzo surfaces
$\tilde{\mathbb P}^2_r$, $4\leq r\leq 8$, from the
Donaldson-Thomas invariants of toric del Pezzo surfaces
$\tilde{\mathbb P}^2_r$, $1\leq r\leq 3$.
\end{remark}

{\bf Acknowledgements} The author would like to thank  Prof.
Yongbin Ruan, Wei-Ping Li, Zhenbo Qin and M. Roth for their
valuable discussions. Thanks also to Dr. P. Li for his help in
drawing the figures.

\section{\bf Preliminaries}
In this section, we shall discuss the basic materials on
Donaldson-Thomas invariants studied by Maulik, Nekrasov, Okounkov
and Pandharipande.
 For
the details, one can consult \cite{D-T, L-R, MNOP1, MNOP2,
Thomas}.

Let $X$ be a smooth projective  3-fold and $\mathcal I$ be an
ideal sheaf on $X$. Assume the sub-scheme $Y$ defined by $\mathcal
I$ has dimension $\le 1$. Here $Y$ is allowed to have embedded
points on the curve components. Therefore we have the exact
sequence
\begin{eqnarray*}
      0\longrightarrow {\mathcal I} \longrightarrow {\mathcal
      O}_X\longrightarrow {\mathcal O}_Y\longrightarrow 0.
\end{eqnarray*}

The $1$-dimensional components, with multiplicities taken into
consideration, determine a homology class
\begin{eqnarray*}
     [Y]\in H_2(X, \Z).
\end{eqnarray*}

Let $I_n(X,\beta)$ denote the moduli space of ideal sheaves
$\mathcal I$ satisfying
\begin{eqnarray*}
     \chi({\mathcal O}_Y) = n,
     \quad
   [Y] = \beta \in H_2(X, \Z).
\end{eqnarray*}
$I_n(X,\beta)$ is projective and is a  fine moduli space. From the
deformation theory, one can compute the virtual dimension of
$I_n(X,\beta)$ to obtain the following result

\begin{lemma}\label{lem2.1}  The virtual dimension of
$I_n(X,\beta)$, denoted by $\text{vdim}$,  equals $\int_\beta
c_1(T_X)$.
\end{lemma}

Note that the actual dimension of the moduli space $I_n(X,\beta)$
is usually larger than the virtual dimension.

 Let $\mathfrak I$ be the universal family
over $I_n(X,\beta)\times X$ and $\pi_i$ be the projection of
$I_n(X,\beta)\times X$ to the $i$-th factor. For a cohomology
class $\gamma \in H^l(X,\Z)$, consider the  operator
\begin{eqnarray*}
    ch_{k+2}(\gamma) : H_*(I_n(X,\beta), {\mathbb Q}) \longrightarrow
    H_{*-2k+2-l}(I_n(X,\beta),{\mathbb Q}),
\end{eqnarray*}
\begin{eqnarray*}
  ch_{k+2}(\gamma)(\xi) = \pi_{1*}(ch_{k+2}({\mathcal J})\cdot
  \pi_2^*(\gamma)\cap\pi_1^*(\xi)).
\end{eqnarray*}

Descendent fields in Donaldson-Thomas theory are defined in
\cite{MNOP2}, denoted by $\tilde{\tau}_k(\gamma)$, which
correspond to the operations $(-1)^{k+1}ch_{k+2}(\gamma)$. The
descendent invariants are defined  by
\begin{eqnarray*}
<\tilde{\tau}_{k_1}(\gamma_{l_1})\cdots
\tilde{\tau}_{k_r}(\gamma_{l_r})>_{n,\beta} =
\int_{[I_n(X,\beta)]^{vir}}\prod_{i=1}^r(-1)^{k_i+1}ch_{k_i+2}(\gamma_{l_i}),
\end{eqnarray*}
where the latter integral is the push-forward to a point of the
class
\begin{eqnarray*}
   (-1)^{k_1+1}ch_{k_1+2}(\gamma_{l_1})\circ \cdots \circ
   (-1)^{k_r+1}ch_{k_r+2}(\gamma_{l_r})([I_n(X,\beta)]^{vir}).
\end{eqnarray*}

The Donaldson-Thomas partition function with descendent insertions
is defined by
\begin{eqnarray*}
   Z_{DT}(X;q\mid\prod_{i=1}^r\tilde{\tau}_{k_i}(\gamma_{l_i}))_\beta
   = \sum_{n\in
   \Z}<\prod_{i=1}^r\tilde{\tau}_{k_i}(\gamma_{l_i})>_{n,\beta}q^n.
\end{eqnarray*}

The degree 0 moduli space $I_n(X,0)$ is isomorphic to the Hilbert
scheme of $n$ points on $X$. The degree 0 partition function is
${\bf Z}_{DT}(X;q)_0$.

The reduced partition function is obtained by formally removing
the degree $0$ contributions,
\begin{eqnarray*}
  Z'_{DT}(X;q\mid\prod_{i=1}^r\tilde{\tau}_{k_i}(\gamma_{l_i}))_\beta
  =
\frac{Z_{DT}(X;q\mid\prod\limits_{i=1}^r\tilde{\tau}_{k_i}(\gamma_{l_i}))_\beta}{Z_{DT}(X;q)_0}.
\end{eqnarray*}

Relative Donaldson-Thomas invarints are also defined in
\cite{MNOP2}. Let $S$ be a smooth divisor in $X$. An ideal sheaf
$\mathcal I$ is said to be relative to $S$ if the morphism
\begin{eqnarray*}
\mathcal I\otimes_{\mathcal O_X}\mathcal O_S\rightarrow \mathcal
O_X\otimes _{\mathcal O_X}\mathcal O_S
\end{eqnarray*}
is injective. A proper moduli space $I_n(X/S,\beta)$ of relative
ideal sheaves can be constructed by considering the ideal sheaves
relative to the expended pair $(X[k], S[k])$. For details, one can
read  \cite{Li2} and \cite{MNOP2}.

Let $Y$ be the subscheme defined by $\mathcal I$. The scheme
theoretic intersection $Y\cap S$ is an element in the Hilbert
scheme of points on $S$ with length $[Y]\cdot S$. If we use
$\mbox{Hilb}(S, k)$ to denote the Hilbert scheme of points of
length $k$ on $S$, we  have a map
\begin{eqnarray*}
   \epsilon : I_n(X/S,\beta) \longrightarrow \mbox{Hilb}(S,
   \beta\cdot [S]).
\end{eqnarray*}

The cohomology of the Hilbert scheme of points of $S$ has a basis
via the representation of the Heisenberg algebra on the
cohomologies of the Hilbert schemes.

Following Nakajima in \cite{Nakajima}, let $\eta$ be a cohomology
weighted partition with respect to a basis of $H^*(S, {\mathbb
Q})$. Let $\eta=\{\eta_1, \ldots, \eta_s\}$ be a partition whose
corresponding cohomology classes are $\delta_1, \cdots, \delta_s$,
let
\begin{eqnarray*}
    C_\eta
    =\frac{1}{\mathfrak{z}(\eta)}P_{\delta_1}[\eta_1]\cdots P_{\delta_s}[\eta_s]\cdot{\bf
    1}\in H^*(\mbox{Hilb}(S, |\eta|), {\mathbb Q}),
\end{eqnarray*}
where
\begin{eqnarray*}
     \mathfrak{z}(\eta) = \prod_i \eta_i|\mbox{Aut}(\eta)|,
\end{eqnarray*}
and $|\eta|=\sum_j\eta_j$. The Nakajima basis of the cohomology of
$\mbox{Hilb}(S,k)$ is the set,
\begin{eqnarray*}
   \{C_\eta\}_{|\eta|=k}.
\end{eqnarray*}

We can choose a basis of $H^*(S)$ so that it is self dual with
respect to the Poincar\'e pairing, i.e., for any $i$, $\delta_i^*
=\delta_j$ for some $j$.  To each weighted partition $\eta$, we
define the dual partition $\eta^\vee$ such that
$\eta^\vee_i=\eta_i$ and the corresponding cohomology class to
$\eta^\vee_i$ is $\delta_i^*$. Then we have
\begin{eqnarray*}
 \int_{\mbox{Hilb}(S,k)}C_\eta\cup C_\nu =
 \frac{(-1)^{k-\ell(\eta)}}{\mathfrak{z}(\eta)}\delta_{\nu,\eta^\vee},
\end{eqnarray*}
see \cite{Nakajima}.

The descendent invariants in the relative Donaldson-Thomas theory
are defined by
\begin{eqnarray*}
<\tilde{\tau}_{k_1}(\gamma_{l_1})\cdots\tilde{\tau}_{k_r}(\gamma_{l_r})\mid
\eta>_{n,\beta} =
\int_{[I_n(X/S,\beta)]^{vir}}\prod_{i=1}^r(-1)^{k_i+1}ch_{k_i+2}(\gamma_{l_i})\cap
\epsilon^*(C_\eta),
\end{eqnarray*}

Define the associated partition function by
\begin{eqnarray*}
   Z_{DT}(X/S;q\mid\prod_{i=1}^r\tilde{\tau}_{k_i}(\gamma_{l_i}))_{\beta,\eta}
   = \sum_{n\in\Z}
   <\prod_{i=1}^r\tilde{\tau}_{k_i}(\gamma_{l_i})\mid\eta>_{n,\beta}q^n.
\end{eqnarray*}

The reduced partition function is obtained by formally removing
the degree $0$ contributions,
\begin{eqnarray*}
Z'_{DT}(X/S;q\mid\prod_{i=1}^r\tilde{\tau}_{k_i}(\gamma_{l_i}))_{\beta,\eta}=
\frac{Z_{DT}(X/S;q\mid\prod\limits_{i=1}^r\tilde{\tau}_{k_i}(\gamma_{l_i}))_{\beta,\eta}}{Z_{DT}(X/S;q)_0}.
\end{eqnarray*}

Since this is the main tool employed in this paper, so in the
remaining of the section, we shall review some notations in the
degeneration formula, see \cite{Li2} for the details.

 Let $\pi\colon\mathcal X\to C$ be
a smooth $4$-fold over a smooth irrreducible curve $C$ with a
marked point denoted by $\bf 0$ such that $\mathcal
X_t=\pi^{-1}(t)\cong X$ for $t\neq {\bf 0}$ and $\mathcal X_{\bf
0}$ is a union of two smooth $3$-folds $X_1$ and $X_2$
intersecting transversely along a smooth surface $S$. We write
$\mathcal X_{\bf 0}=X_1\cup_SX_2$. Assume that $C$ is contractible
and $S$ is simply-connected.

Consider the natural maps
\begin{eqnarray*}
i_t\colon X=\mathcal X_t\rightarrow \mathcal X,\qquad i_{\bf
0}\colon \mathcal X_{\bf 0}\rightarrow \mathcal X,
\end{eqnarray*}
and the gluing map
\begin{eqnarray*}
g=(j_1, j_2)\colon X_1\coprod X_2\rightarrow \mathcal X_{\bf 0}.
\end{eqnarray*}

We have
\begin{eqnarray*}
H_2(X){\buildrel{i_{t*}}\over \longrightarrow} H_2(\mathcal
X){\buildrel{i_{0*}}\over \longleftarrow} H_2(\mathcal X_{\bf
0}){\buildrel{g_*}\over\longleftarrow} H_2(X_1)\oplus H_2(X_2),
\end{eqnarray*}
where $i_{0*}$ is an isomorphism since there exists a deformation
retract from $\mathcal X$ to $\mathcal X_{\bf 0}$ (see
\cite{Clemens}) and $g_*$ is surjective from Mayer-Vietoris
sequence. For $\beta \in H_2(X)$, there exist $\beta_1\in
H_2(X_1)$ and $\beta_2\in H_2(X_2)$ such that
\begin{eqnarray}\label{betasum}
i_{t*}(\beta)=i_{0*}(j_{1*}(\beta_1)+j_{2*}(\beta_2)).
\end{eqnarray}
For simplicity, we write $\beta=\beta_1+\beta_2$ instead.
\begin{lemma}\label{lem2.2}
With the assumption as above,  given $\beta = \beta_1 + \beta_2$.
Let $ d = \int_\beta c_1(X)$ and $d_i = \int_{\beta_i}c_1(X_i)$, $
i = 1,2$. Then
\begin{eqnarray}\label{deg-formula}
    d = d_1 + d_2 - 2\int_{\beta_1}[S],\qquad  \int_{\beta_1} [S]=\int_{\beta_2}
    [S].
\end{eqnarray}
\end{lemma}
\begin{proof}
The formulae (\ref{deg-formula}) come from the adjunction formulae
$K_{\mathcal X_t}=K_{\mathcal X}|_{\mathcal X_t}$ and
$K_{X_i}=(K_{\mathcal X}+X_i)|_{X_i}$ for $i=1, 2$, and $X_1\cdot
(X_1+X_2)=X_1\cdot \mathcal X_{\bf 0}=0$.
\end{proof}

 Similarly for cohomology, we have the maps
\begin{eqnarray*}
H^k(\mathcal X_t){\buildrel{i_t^*}\over\longleftarrow}
H^k(\mathcal X){\buildrel{i_0^*}\over \longrightarrow }
H^k(\mathcal X_{\bf
0}){\buildrel{g^*}\over\longrightarrow}H^k(X_1)\oplus H^k(X_2),
\end{eqnarray*}
where $i_0^*$ is an isomorphism. Take $\alpha\in H^k(\mathcal X)$
and let $\alpha(t)=i^*_t\alpha$.

There is a degeneration formula  which takes the form
\begin{eqnarray}\label{degeneration-formula}
& &Z'_{DT}(\mathcal
X_t;q\mid\prod\limits_{i=1}^r\tilde{\tau}_0(\gamma_{l_i}(t)))_\beta\nonumber\\
& = &\sum
Z'_{DT}({X_1}/{S};q\mid\prod\tilde{\tau}_0(j_1^*\gamma_{l_i}(0)))_{\beta_1,\eta}\displaystyle{
\frac{(-1)^{|\eta|-\ell(\eta)}\mathfrak{z}(\eta)}{q^{|\eta|}}}
\\
& &\times Z'_{DT}({X_2}/{S};q\mid\prod
\tilde{\tau}_0(j_2^*\gamma_{\l_i}(0)))_{\beta_2,\eta^\vee},\nonumber
\end{eqnarray}
where the sum is over the splittings $\beta_1 + \beta_2 = \beta$,
and cohomology weighted partitions $\eta$. $\gamma_{l_i}$'s are
cohomology classes on $\mathcal X$. There is a compatibility
condition
\begin{eqnarray}\label{comp}
|\eta|=\beta_1\cdot [S]=\beta_2\cdot [S].
\end{eqnarray}

For details, one can see \cite{Li1, Li2, MNOP2}.

\section{Projective completion  }\label{pc}
   In this section, we describe how to obtain $Y_{\tilde{S}}$ from
$Y_S$ by the degenerations. This makes it possible to find some
relations between the local Donaldson-Thomas invariants of
$\tilde{S}$ and $S$.

   Let $S$ be a smooth surface and $Y_S= \mathbb{P}(K_S\oplus
{\mathcal O})$ the projective completion of its canonical bundle
$K_S$. Consider the blowup $p: \tilde{S}\rightarrow S$ of $S$ at a
smooth point $p_0$ and denote by  $E$ the exceptional divisor in
$\tilde{S}$. Since $Y_S$ is the bundle $\mathbb{P}(K_S\oplus
{\mathcal O})$ over $S$, one can pull this bundle back to
$\tilde{S}$ using the projection $p$. It is easy to see that the
pullback bundle is the same thing as blowing up the fiber over
$p_0$. Denote by $\tilde{Y}_S$ the blowup of $Y_S$ along the fiber
$F_{p_0}\cong \mathbb{P}^1$ over $p_0$, and the exceptional
divisor in $\tilde{Y}_S$ is denoted by $D_1:= E\times \mathbb{P}^1
= \mathbb{P}_{\mathbb{P}^1}({\mathcal O}\oplus {\mathcal O})$. In
$\tilde{Y}_S$, take a section, $\sigma$ , corresponding to
${\mathcal O}\longrightarrow {\mathcal O}\oplus K_S$, of the
exceptional divisor $D_1$ over $E$ and blow it up. Denote by $Z$
the blown-up manifold,then $Z$ has a natural projection $\pi$ to
$\tilde{S}$ given by the composition of the blowup projection
$Z\longrightarrow \tilde{Y}_S$ and the bundle projection
$\tilde{Y}_S\longrightarrow \tilde{S}$. It is easy to see that the
fiber $\pi^{-1}(E)$ has two normal crossing components: $D_1\cong
\mathbb{F}_0$ and $D_2\cong \mathbb{F}_1$ intersecting along a
section $\sigma$ with the normal bundle
$N_{\sigma|\mathbb{F}_0}\cong {\mathcal O}$ and
$N_{\sigma|\mathbb{F}_1}\cong {\mathcal O}(-1)$ respectively.

Next, we consider the projective completion $Y_{\tilde{S}}$. Since
the restriction $K_{\tilde{S}}\mid_E$ of the canonical bundle
$K_S$ to the exceptional divisor $E$ in $\tilde{S}$ is isomorphic
to ${\mathcal O}(-1)$, so we can pick up a section, $\sigma_1$, of
the restriction of $Y_{\tilde{S}}$ to $E$ satisfying $\sigma_1^2
=-1$. Then we blow this section $\sigma_1$ up, and it is easy to
know that the blown-up manifold is $Z$. Here we illustrate the
sequence of birational maps by Figure 1.

\begin{figure}[] \centering
\includegraphics[bb=0 0 505 300,totalheight=90mm,
angle=0,width=150mm]{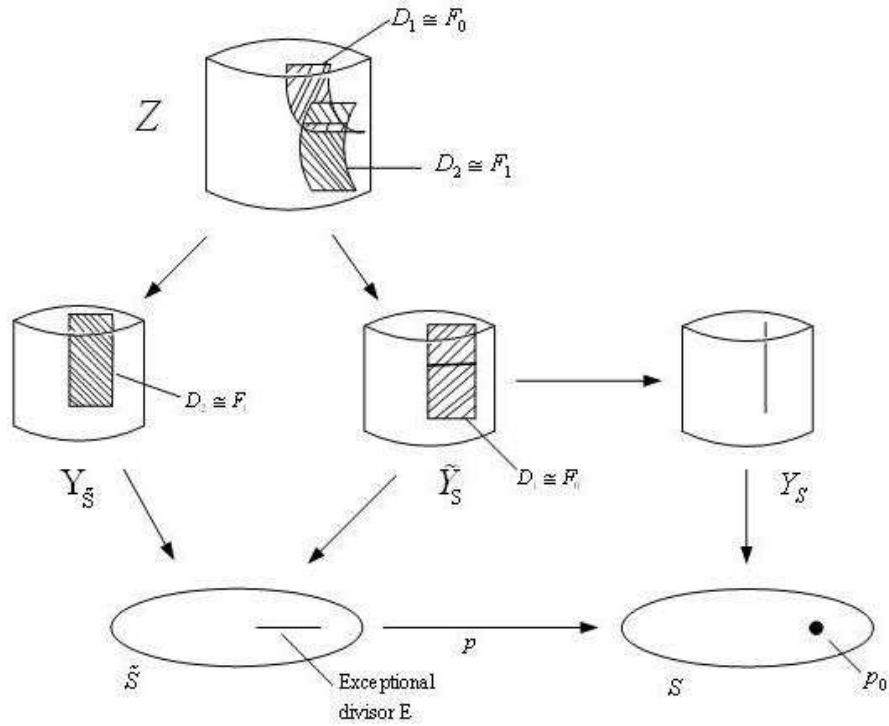} \caption{  birational maps }
\end{figure}

Let $\tilde{\mathbb P}^2_{r}$ be the blowup of ${\mathbb P}^2$ at
r points. Pick one more point $p$ and blow it up, then we obtain
$\tilde{\mathbb P}^2_{r+1} $ with the map $p:\tilde{\mathbb
P}^2_{r+1}\longrightarrow \tilde{\mathbb P}^2_r$ and denote by $E$
the exceptional divisor in $\tilde{\mathbb P}^2_{r+1}$. It is
well-known that for $0\leq r \leq 3$, $\tilde{\mathbb{P}}_r^2$ is
toric, but for $4\leq r\leq 8$, $\tilde{\mathbb P}^2_r$ is
non-toric. In \cite{MNOP1}, via the localization technique, the
authors computed the local Donaldson-Thomas invariants of toric
surfaces, in particular, their method is valid for del Pezzo
surfaces $\tilde{\mathbb{P}}_r^2$ with $0\leq r \leq 3$. As
opposed to toric del Pezzo surfaces, one can not directly use
localization with respect to a torus action because there is no
torus action on a generic del Pezzo surface $\tilde{\mathbb
P}^2_r$,$4\leq r\leq 8$. Our Theorem \ref{thm-1-1} implies that
for some degrees, we could compute the local Donaldson-Thomas
invariants of non-toric surfaces $\tilde{\mathbb P}^2_r$ with
$4\leq r\leq 8$ from the local Donaldson-Thomas invariants of
$\tilde{\mathbb P}^2_r$ with $0\leq r\leq 3$.

\section{Main theorems}

Using the notation as before, we have

\begin{lemma}\label{lem4.1}
Suppose that $S$ is a smooth surface. Let $\tilde{Y}_S$ be the
blowup of $Y_S$ along the fiber over $p_0\in S$. Then for any
$\beta\in H_2(S;{\mathbb Z})$, we have
$$
   Z'_{DT}(Y_S;q)_\beta = Z'_{DT}(\tilde{Y}_S/D_1; q)_{p!(\beta), \emptyset},
$$
where $D_1={\mathbb P}_{{\mathbb P}^1}({\mathcal O}\oplus
{\mathcal O})\cong {\mathbb P}^1\times {\mathbb P}^1$ is the
exceptional divisor in $\tilde{Y}_S$, $p: \tilde{S}\longrightarrow
S$ is the blowup of $S$ at $p_0$ and $p!(\beta) = PD
p^*PD(\beta)$.
\end{lemma}

\begin{proof}
Let $\mathcal X$ be the blow up of $Y_S\times {\mathbb C}$ along
$F_{p_0}\times \{0\}$, where $F_{p_0}$ is the fiber of $Y_S$ over
$p_0$ and let $\pi$ be the natural projection from $\mathcal X$ to
$\mathbb C$. It is a semistable degeneration of $Y_S$ with the
central fiber ${\mathcal X}_0$ being a union of $X_1= \tilde{Y}_S$
and $X_2\cong {\mathbb P}_{{\mathbb P}^1}({\mathcal O}\oplus
{\mathcal O}\oplus {\mathcal O}) \cong {\mathbb P}^2\times
{\mathbb P}^1$ with the common divisor $D_1\cong {\mathbb
P}^1\times {\mathbb P}^1$.

By the degeneration formula (\ref{degeneration-formula}), we may
express the absolute Donaldson-Thomas invariants of $Y_S$ in terms
of the relative Donaldson-Thomas invariants of $(X_1,D_1)$ and
$(X_2,D_1)$ as follows:
\begin{eqnarray}
   Z'_{DT}(Y_S;q)_\beta &=& \sum_{\eta, \beta_1+\beta_2
   =\beta}Z'_{DT}(Y_S/D_1;q)_{\beta_1,\eta}\\
   & & \times \frac{(-1)^{|\eta| - \ell(\eta)}{\mathfrak
   z}(\eta)}{q^{|\eta|}}Z'_{DT}(X_2/D_1;q)_{\beta_2,\eta^\vee}\nonumber
\end{eqnarray}
where the summation runs over the splittings $\beta_1+\beta_2 =
\beta$ and the cohomology weighted partitions $\eta$.

Now we need to compute the summands in the right hand side of the
degeneration formula. For this we have the following claim:

{\bf Claim:} There are only terms with $\beta_2=0$.

In fact, if $|\eta|\not= 0$, then $\beta_2\not= 0$ because
$\beta_2\cdot D_1 = |\eta|$. By Lemma \ref{lem2.1}, we have
$$
  C_1(X_1)\cdot \beta_1 = \mbox{vdim} I_n(X_1/D_1; \beta_1) = \deg
  \epsilon^*_1(C_\eta),
$$
where $C_1(X_1)$ denotes the first Chern class of $X_1$ and
$\epsilon_1: I_n(X_1/D_1,\beta_1) \rightarrow
\mbox{Hilb}(D_1,|\eta|)$ is the canonical intersection map.

Let $V$ be a complex rank $r$ vector bundle over a complex
manifold $M$, and $\pi: {\mathbb P}(V)\longrightarrow M$ be the
corresponding projective bundle. Let $\xi_V$ be the first Chern
class of the tautological bundle in ${\mathbb P}(V)$. A simple
calculation shows
\begin{equation}\label{chernclass}
              C_1({\mathbb P}(V)) = \pi^* C_1(M) + \pi^*C_1(V) -
              r\xi_V.
\end{equation}

Applying (\ref{chernclass}) to $X_2$, we obtain
$$
      C_1(X_2) = \pi^*{\mathcal O}_{{\mathbb P}^1}(2) - 3\xi,
$$
where $\xi$ is the first Chern class of the tautological bundle in
$X_2$. Since the homology class $\beta_2$ may be decomposed into
the sum of the base class $\beta_2^{{\mathbb P}^1}$ and the fiber
class $\beta_2^f$, so we have
\begin{eqnarray*}
   C_1(X_2)\cdot \beta_2 &=& \mbox{vdim} I_n(X_2/D_1,\beta_2)\\
    & =& \pi^*{\mathcal O}_{{\mathbb P}^1}(2)\cdot\beta_2^{{\mathbb
   P}^1} - 3\xi\cdot \beta_2^f \geq 3|\eta|.
\end{eqnarray*}

In the last inequality, we use the fact that $-\xi$ is the
infinite section, so $-\xi\cdot\beta_1^f = |\eta|$ and
$\pi^*{\mathcal O}_{{\mathbb P}^1}(2)\beta_2^{{\mathbb P}^1} =
{\mathcal O}_{{\mathbb P}^1}(2)\cdot\pi_*\beta_2^{{\mathbb
P}^1}\geq 0$. Since $\beta\in H_2(S,{\mathbb Z})$, from
(\ref{chernclass}), we have
$$
    c_1(Y_S)\cdot\beta =\mbox{vdim} I_n(X,\beta) =0.
$$

From Lemma \ref{lem2.2}, we have
$$
   C_1(Y_S)\cdot\beta = c_1(X_1)\cdot\beta_1 +
   C_1(X_2)\cdot\beta_2 - 2|\eta|.
$$
Therefore, we obtian
$$
     \deg C_\eta + |\eta|>0.
$$
This is a contradiction. Therefore $|\eta|=0$. So the claim is
proved.

Thus $\beta_2\cdot D_1 =0$. Since $D_1$ is the hyperplane in
$X_2\cong {\mathbb P}^3$, we must have $\beta_2 =0$. Also we have
$\beta_1= p!(\beta)$.

By the degeneration formula, we have
\begin{equation*}
  Z'_{DT}(Y_S;q)_\beta = Z'_{DT}(\tilde{Y}_S/D_1;q)_{p!(\beta),\emptyset}.
\end{equation*}
This proves the lemma.
\end{proof}

\begin{lemma}\label{lem4.2}
Under the assumption of Lemma \ref{lem4.1}, Then for $\beta\in
H_2(S;{\mathbb Z})$, we have
$$
      Z'_{DT}(\tilde{Y_S};q)_{p!(\beta)} =
      Z'_{DT}(\tilde{Y_S}/D_1;
      q)_{p!(\beta), \emptyset}
$$
\end{lemma}

\begin{proof}

   Let $\mathcal X$ be the blow up of $\tilde{Y}_S\times {\mathbb C}$ along
$D_1\times \{0\}$. Let $\pi: {\mathcal X}\longrightarrow {\mathbb
C}$ be the natural projection. Thus we get a semi-stable
degeneration of $\tilde{Y}_S$ whose central fiber is a union of
$X_1\cong \tilde{Y}_S$ and $X_2= {\mathbb P}_{{\mathbb
P}^1}(N_{D_1}\oplus {\mathcal O})$, where the normal bundle of the
divisor $D_1$ is $N_{D_1}= {\mathcal O}(-1,-1)$.

By the degeneration formula (\ref{degeneration-formula}), we may
express the absolute Donaldson-Thomas invariants of $\tilde{Y}_S$
in terms of the relative Donaldson-Thomas invariants of
$(X_1,D_1)$ and $(X_2,D_1)$ as follows:
\begin{eqnarray}\label{4-2-1}
   Z'_{DT}(\tilde{Y}_S;q)_{p!(\beta)} &=& \sum_{\eta, \beta_1+\beta_2
   =\beta}Z'_{DT}(\tilde{Y}_S/D_1;q)_{\beta_1,\eta}\\
   & & \times \frac{(-1)^{|\eta| - \ell(\eta)}{\mathfrak
   z}(\eta)}{q^{|\eta|}}Z'_{DT}(X_2/D_1;q)_{\beta_2,\eta^\vee}\nonumber
\end{eqnarray}
where the summation runs over the splittings $\beta_1+\beta_2 =
\beta$ and the cohomology weighted partitions $\eta$.

 Similar to the proof of Lemma \ref{lem4.1}, we need to prove that there are only terms with
$\beta_2=0$ in the right hand side of the degeneration formula.

In fact, Note that $X_2 = {\mathbb P}_{D_1}(N_{D_1}\oplus
{\mathcal O})$ and $D_1={\mathbb P}_{{\mathbb P}^1}({\mathcal
O}\oplus {\mathcal O})$. Denote by $F_{p_0}\cong {\mathbb P}^1$
the fiber of $Y_S$ at the point $p_0$. Applying (\ref{chernclass})
to $X_2$ and $D_1$, we obtain
\begin{eqnarray*}
               C_1(X_2) &=& \pi^*C_1(D_1) + \pi^*C_1(N_{D_1}) - 2\xi\\
               &=& \pi^*C_1(F_{p_0}) + \pi^*C_1(N_{F_{p_0}|Y_S})- 2\xi_1 +\pi^*C_1(N_{D_1})
               - 2\xi,
\end{eqnarray*}
where $\xi_1$ and $\xi$ are the first Chern classes of the
tautological bundles in ${\mathbb P}(N_{F_{p_0}|Y_S})$ and
${\mathbb P}(N_{D_1}\oplus {\mathcal O})$ respectively. Here we
denote the Chern class and its pullback by the same symbol.  It is
well-known that the normal bundle to $D_1$ in $\tilde{Y}_S$ is
just the tautological line bundle on $D_1\cong {\mathbb
P}(N_{F_{p_0}|Y_S})$. Therefore $C_1(N_{D_1}) = \xi_1$. So we have
$$
    C_1(X_2) = \pi^*C_1(F_{p_0}) - \xi_1 - 2\xi.
$$
where $-\xi$ is the infinite section which has positive
intersections with the effective curve classes.

 Note that $X_2$ is a projective bundle over $D_1$ with fiber ${\mathbb
 P}^1$.  Let $L$ be the class of a line in the fiber ${\mathbb
 P}^1$ and $e$ be the class of a line in the fiber ${\mathbb P}^1$
 in $D_1= {\mathbb P}(N_{F_{p_0}|Y_S})$. Denote by $\beta_2^{F_{p_0}}$ the
 homology class of the projection in $F_{p_0}$ of the curve component.
 Denote by $\beta_2^f$ the difference of $\beta_2$ and
 $\beta_2^{F_{p_0}}$, i. e. $\beta_2^f = \beta_2-\beta^{F_{p_0}}$. Then
 it is easy to know $\beta_2^f = aL +be$. Since $(-\xi)\cdot
 \beta_2 = |\eta|$  and
 $(-\xi)\cdot \beta_2^f = a = |\eta|$. On the other hand,
 since all curves of class $\beta_2$ come from the curve of class $p!(\beta)$
 by the degeneration and the degeneration only happens away from
 the divisor $D_1$, from $p!(\beta)\cdot D_1 =0$, we have
 $D_1\cdot \beta_2 =0$. Thus we have $D_1\cdot \beta_2^f = a-b =0$.
 Therefore, we have $a=b=|\eta|$. So we have $\beta_2^f = |\eta|(L+e)$.
 Since $C_1(F_{p_0}) + C_1(N_{F_{p_0}|Y_S})=C_1(F_{p_0})\geq 0$, we have
 $$
      C_1(X_2)\cdot\beta_2 \geq  4|\eta|.
 $$

 From Lemma \ref{lem2.2}, we have
 $$
   C_1(\tilde{Y}_S)\cdot p!(\beta) = C_1(X_1)\cdot\beta_1 +
   C_1(X_2)\cdot\beta_2 -2|\eta|.
 $$
 Therefore,
 $$
  \deg C_\eta + 2|\eta| >0.
 $$
 This is a contradiction. Thus $|\eta| =0$.

 Therefore, from the discussion above, we have  $\beta_2=
 \beta^{F_{p_0}}$. So $C_1(X_2)\cdot \beta_2 =
 C_1(F_{p_0})(\beta^{F_{p_0}})$. Thus $C_1(X_2)\cdot \beta_2 >0 $
 if $\beta^{F_{p_0}}\not= 0$. Furthermore, if $\beta_2= \beta^{F_{p_0}}\not=
 0$, then, by definition, we have
 $$
    Z'_{DT}(X_2/D_1; q)_{\beta_2,\emptyset} = 0.
 $$
 Therefore, we have proved that there are only terms with
 $\beta_2=0$ in the right hand side of (\ref{4-2-1}).

 By the degeneration formula, we have
 \begin{equation*}
  Z'_{DT}(\tilde{Y}_S;q)_{p!(\beta)} =
  Z'_{DT}(\tilde{Y}_S/D_1;q)_{p!(\beta),\emptyset}.
 \end{equation*}
 This proves the lemma.

\end{proof}

Summarizing Lemma \ref{lem4.1} and \ref{lem4.2}, we have

\begin{theorem}
 $$
    Z'_{DT}(Y_S;q)_\beta = Z'_{DT}(\tilde{Y_S}; q)_{p!(\beta)}.
 $$
\end{theorem}

Next, we want to compare the Donaldson-Thomas invariants
$Z'_{DT}(\tilde{Y}_S;q)_{p!(\beta)}$ of $\tilde{Y}_S$ to the
Donaldson-Thomas invariants of $Z$. In fact, we have

\begin{theorem}\label{thm-4-4}
$$
  Z'_{DT}(\tilde{Y}_S;q)_{p!(\beta)} = Z'_{DT}(Z;q)_{p!(\beta)}.
$$
\end{theorem}

\begin{proof}
 In $\tilde{Y}_S$, take a section $\sigma$ of the exceptional
 divisor $D_1$ over the old exceptional divisor $E$, then $\sigma \cong {\mathbb
 P}^1$ and the normal bundle to $\sigma$ in $\tilde{Y}_S$ is
 $N_{\sigma|\tilde{Y}_S} = {\mathcal O}_{{\mathbb P}^1}(-1)\oplus{\mathcal O}$.
 Let $\mathcal X$ be the blow up of $\tilde{Y}_S\times {\mathbb
 C}$ along $\sigma\times \{0\}$. Let $\pi : {\mathcal X}\longrightarrow {\mathbb
 C}$ be the natural projection. Thus we get a semi-stable
 degeneration of $\tilde{Y}_S$ whose central fiber is a union of
 $X_1 \cong Z$ and $X_2 = {\mathbb P}_\sigma ({\mathcal O}_{{\mathbb P}^1}(-1)\oplus {\mathcal O}\oplus {\mathcal
 O})$ with the Hirzebruch surface ${\mathbb F}_1= {\mathbb P}_\sigma({\mathcal O}(-1)\oplus {\mathcal
 O})$ as the common divisor.

    By the degeneration formula (\ref{degeneration-formula}), we may
express the absolute Donaldson-Thomas invariants of $\tilde{Y}_S$
in terms of the relative Donaldson-Thomas invariants of
$(X_1,{\mathbb F}_1)$ and $(X_2,{\mathbb F}_1)$ as follows:
\begin{eqnarray}\label{4-4-1}
   Z'_{DT}(\tilde{Y}_S;q)_{p!(\beta)} &=& \sum_{\eta, \beta_1+\beta_2
   =\beta}Z'_{DT}(Z/{\mathbb F}_1;q)_{\beta_1,\eta}\\
   & & \times \frac{(-1)^{|\eta| - \ell(\eta)}{\mathfrak
   z}(\eta)}{q^{|\eta|}}Z'_{DT}(X_2/{\mathbb F}_1;q)_{\beta_2,\eta^\vee}\nonumber
\end{eqnarray}
where the summation runs over the splittings $\beta_1+\beta_2 =
p!(\beta)$ and the cohomology weighted partitions $\eta$.

Similar to the proof of Lemma \ref{lem4.1}, we need to prove that
there are only terms with $\beta_2=0$ in the right hand side of
(\ref{4-4-1}).

Note that $X_2={\mathbb P}_\sigma ({\mathcal O}_{{\mathbb
P}^1}(-1)\oplus {\mathcal O}\oplus {\mathcal O})$ and ${\mathbb
F}_1= {\mathbb P}_\sigma({\mathcal O}(-1)\oplus {\mathcal
 O})$. Therefore, from (\ref{chernclass}), we have
 $$
     C_1(X_2) = \pi^*C_1({\mathcal O}_\sigma (1)) - 3\xi
 $$
where $\xi$ is the first Chern class of the tautological line
bundle over $X_2$. It is easy to see that
$$
   C_1(X_2)\cdot \beta_2 \geq 3|\eta|.
$$

From Lemma \ref{lem2.2}, we have
$$
   C_1(\tilde{Y}_S)\cdot p!(\beta)=C_1(Z)\cdot \beta_1 +
   C_1(X_2)\cdot \beta_2 - 2|\eta|.
$$
Therefore,
$$
    C_1(\tilde{Y}_S)\cdot p!(\beta)\geq \deg C_\eta + |\eta|>0.
$$
Since $C_1(\tilde{Y}_S)\cdot p!(\beta) =0$, so this is a
contradiction. Thus $|\eta| =0$.

The same argument as in the proof of Lemma \ref{lem4.2} shows that
$\beta_2=0$. Therefore, by the degeneration formula, we have
\begin{equation}\label{4-4-2}
Z'_{DT}(\tilde{Y}_S;q)_{p!(\beta)} = Z'_{DT}(Z/{\mathbb F}_1;
q)_{p!(\beta),\emptyset}.
\end{equation}

Now it remains to prove
$$
    Z'_{DT}(Z;q)_{p!(\beta)} = Z'_{DT}(Z/{\mathbb
    F}_1;q)_{p!(\beta),\emptyset}.
$$
To prove this, we degenerate $Z$ along the exceptional divisor
${\mathbb F}_1$. Then we obtain two smooth $3$-folds
$$
   X_1 = Z, \hspace{2cm} X_2 = {\mathbb P}_{{\mathbb F}_1}(N_{{\mathbb F}_1}\oplus {\mathcal
   O}),
$$
intersecting along the exceptional divisor ${\mathbb F}_1$ in $Z$
and the infinite section of the ${\mathbb P}^1$-bundle $X_2$.

Applying the degeneration formula to $Z'_{DT}(Z;q)_{p!(\beta)}$,
we have
\begin{eqnarray}\label{4-4-3}
   Z'_{DT}(Z;q)_{p!(\beta)} &=& \sum_{\eta, \beta_1+\beta_2
   =\beta}Z'_{DT}(Z/{\mathbb F}_1;q)_{\beta_1,\eta}\\
   & & \times \frac{(-1)^{|\eta| - \ell(\eta)}{\mathfrak
   z}(\eta)}{q^{|\eta|}}Z'_{DT}(X_2/{\mathbb F}_1;q)_{\beta_2,\eta^\vee}\nonumber
\end{eqnarray}
where the summation runs over the splittings $\beta_1+\beta_2 =
p!(\beta)$ and the cohomology weighted partitions $\eta$.

Note that $X_2 = {\mathbb P}_{{\mathbb F}_1}(N_{{\mathbb
F}_1}\oplus {\mathcal O})$ and ${\mathbb F}_1 = {\mathbb P}_\sigma
({\mathcal O}(-1)\oplus {\mathcal O})$. Applying
(\ref{chernclass}) to $X_2$ and ${\mathbb F}_1$, we obtain
\begin{eqnarray*}
   C_1(X_2) &=& \pi^*C_1({\mathbb F}_1) + \pi^* C_1(N_{{\mathbb
   F}_1}) - 2\xi\\
   &=& \pi^*C_1({\mathcal O}_\sigma (1)) - \xi_1 -2\xi,
\end{eqnarray*}
where $\xi_1$ and $\xi$ are the first Chern classes of the
tautological bundles in ${\mathbb P}_\sigma({\mathcal O}(-1)\oplus
{\mathcal O})$ and ${\mathbb P}(N_{{\mathbb F}_1}\oplus {\mathcal
O})$ respectively. Here we denote the Chern class and its pullback
by the same symbol. The same calculation as in the proof of Lemma
\ref{lem4.2} shows that
$$
   C_1(X_2)\cdot \beta_2 \geq 4|\eta|.
$$

 From Lemma \ref{lem2.2}, we have
 $$
   C_1(Z)\cdot p!(\beta) = C_1(X_1)\cdot\beta_1 +
   C_1(X_2)\cdot\beta_2 -2|\eta|.
 $$
 Therefore, if $|\eta|\not= 0$, then
 $$
  \deg C_\eta + 2|\eta| >0.
 $$
 This is a contradiction because $C_1(Z)\cdot p!(\beta) =0$. Thus $|\eta| =0$.

 The same argument shows that $\beta_2= \beta_\sigma$, i. e. the class of a curve in $\sigma$.
 So $C_1(X_2)\cdot \beta_2 = C_1({\mathcal O}_\sigma(1))(\beta_\sigma)$. Thus $C_1(X_2)\cdot \beta_2 >0 $
 if $\beta_\sigma\not= 0$. Furthermore, if $\beta_2= \beta_\sigma\not=
 0$, then, by definition, we have
 $$
    Z'_{DT}(X_2/D_1; q)_{\beta_2,\emptyset} = 0.
 $$
 Therefore, we have proved that there are only terms with
 $\beta_2=0$ in the right hand side of (\ref{4-4-3}).

 By the degeneration formula, we have
 \begin{equation*}
  Z'_{DT}(Z;q)_{p!(\beta)} =
  Z'_{DT}(Z/{\mathbb F}_1;q)_{p!(\beta),\emptyset}.
 \end{equation*}
 This proves the lemma.

\end{proof}

Finally, we want to prove the following theorem

\begin{theorem}\label{thm-4-5}
$$
        Z'_{DT}(Y_{\tilde{S}};q)_{p!(\beta)} =
    Z'_{DT}(Z;q)_{p!(\beta)}.
$$
\end{theorem}

\begin{proof}
Take a section $\sigma_1\cong {\mathbb P}^1$ of
$Y_{\tilde{S}}\mid_E= {\mathbb F}_1$ such that $\sigma_1^2 = -1$.
Then we degenerate $Y_{ \tilde{S}}$ along the section $\sigma_1$
and obtain two $3$-folds, see Section 3,
$$
 X_1 = Z, \hspace{2cm} X_2 = {\mathbb P}_{\sigma_1}({\mathcal O}(-1)\oplus {\mathcal O}(-1)\oplus {\mathcal
 O}),
$$
with the common divisor ${\mathbb F}_0= {\mathbb
P}_{\sigma_1}({\mathcal O}(-1)\oplus {\mathcal O}(-1))$.

Applying the degeneration formula to
$Z'_{DT}(Y_{\tilde{S}};q)_{p!(\beta)}$, we have
\begin{eqnarray}\label{4-5-1}
   Z'_{DT}(Y_{\tilde{S}};q)_{p!(\beta)} &=& \sum_{\eta, \beta_1+\beta_2
   =\beta}Z'_{DT}(Z/{\mathbb F}_0;q)_{\beta_1,\eta}\\
   & & \times \frac{(-1)^{|\eta| - \ell(\eta)}{\mathfrak
   z}(\eta)}{q^{|\eta|}}Z'_{DT}(X_2/{\mathbb F}_0;q)_{\beta_2,\eta^\vee}\nonumber
\end{eqnarray}
where the summation runs over the splittings $\beta_1+\beta_2 =
p!(\beta)$ and the cohomology weighted partitions $\eta$.

Similar to the proof of Lemma \ref{lem4.1}, we need to prove that
the summand with nonzero contribution in the right hand side of
(\ref{4-5-1}) must have the trivial partition $\eta = \emptyset$.

Note that $X_2 = {\mathbb P}_{\sigma_1}({\mathcal O}(-1)\oplus
{\mathcal O}(-1)\oplus {\mathcal O})$. From (\ref{chernclass}), it
is easy to know
$$
     C_1(X_2) = \pi^*C_1(\sigma_1) + \pi^*C_1({\mathcal
     O}(-1)\oplus {\mathcal O}(-1)) - 3 \xi = -3\xi,
$$
where $\xi$ is the first Chern class of the tautological line
bundle over $X_2$. Therefore, we have
$$
   C_1(X_2)\cdot \beta_2 = 3|\eta|.
$$

From Lemma \ref{lem2.2}, we have
$$
   C_1(Y_{\tilde{S}})\cdot p!(\beta) = C_1(X_1)\cdot \beta_1 +
   C_1(X_2)\cdot \beta_2 -2|\eta|.
$$
Therefore, if $|\eta|\not= 0$, then
$$
C_1(Y_{\tilde{S}})\cdot p!(\beta) = \deg C_\eta + |\eta| >0.
$$
This is a contradiction because $C_1(Y_{\tilde{S}})\cdot p!(\beta)
=0$. This means that the summand with nonzero contribution in the
right hand side of (\ref{4-5-1}) must have $\eta=\emptyset$.

Since the section $\sigma_1$ and the old exceptional divisor $E$
have the same homology class in $Y_{\tilde{S}}$, so from
(\ref{betasum}) and $\eta=\emptyset$, we have $\beta_2=0$.

Therefore, by the degeneration formula, we have
 \begin{equation*}
  Z'_{DT}(Y_{\tilde{S}};q)_{p!(\beta)} =
  Z'_{DT}(Z/{\mathbb F}_0;q)_{p!(\beta),\emptyset}.
 \end{equation*}

Now it remains to prove
 \begin{equation}\label{4-5-2}
  Z'_{DT}(Z;q)_{p!(\beta)} =
  Z'_{DT}(Z/{\mathbb F}_0;q)_{p!(\beta),\emptyset}.
 \end{equation}

To prove this, we degenerate $Z$ along the exceptional divisor
${\mathbb F}_0$. Then we obtain two $3$-folds
$$
     X_1 = Z, \hspace{2cm} X_2 = {\mathbb P}_{{\mathbb F}_0}(N_{{\mathbb F}_0}\oplus {\mathcal
     O}).
$$

Note that $X_2 = {\mathbb P}_{{\mathbb F}_0}(N_{{\mathbb
F}_0}\oplus {\mathcal O})$. Applying (\ref{chernclass}) to $X_2$
and ${\mathbb F}_0$, we have
\begin{eqnarray*}
C_1(X_2) &=& \pi^*C_1({\mathbb F}_0) + \pi^*C_1(N_{{\mathbb F}_0}
- 2\xi \\
& = & \pi^*C_1(\sigma_1) + \pi^*C_1({\mathcal O}(-1)\oplus
{\mathcal O}(-1)) -2\xi_1 +\pi^*C_1(N_{{\mathbb F}_0})-2\xi\\
&=& -\xi_1-2\xi,
\end{eqnarray*}
where $\xi_1$ and $\xi$ are the first Chern classes of the
tautological bundles in ${\mathbb P}_{\sigma_1}({\mathcal
O}(-1)\oplus {\mathcal O}(-1))$ and ${\mathbb P}(N_{{\mathbb
F}_0}\oplus {\mathcal O})$ respectively. The same calculation as
in the proof of Lemma \ref{lem4.2} shows that
$$
     C_1(X_2)\cdot\beta_2 = 4|\eta|.
$$

From Lemma \ref{lem2.2}, we have
\begin{eqnarray*}
   C_1(Z)\cdot p!(\beta) &=& C_1(X_1)\cdot \beta_1 +
   C_1(X_2)\cdot\beta_2 -2|\eta|\\
   &=& \deg C_\eta + 2|\eta| >0.
\end{eqnarray*}
This is a contradiction because $C_1(Z)\cdot p!(\beta)=0$. Thus
$|\eta|=0$.

Since ${\mathbb F}_0\cdot p!(\beta) =0$, the same argument as
above shows that $\beta_2=0$. As before, this implies
(\ref{4-5-2}). This comletes the proof of the theorem.
\end{proof}

\begin{remark}
From Theorem \ref{thm-4-4} and \ref{thm-4-5}, it is easy to know
that Theorem \ref{thm-1-1} holds.
\end{remark}

\end{document}